\documentclass[12pt]{amsart}
\usepackage{amsmath,amscd,amsthm,amsfonts, amssymb,amsxtra, mathrsfs,enumitem,relsize, stackengine}
\usepackage{calrsfs}
\usepackage{caption}
\usepackage[us,12hr]{datetime}
\usepackage[all]{xy} \SelectTips{cm}{}
\usepackage{hyperref}
 \hypersetup{colorlinks=true,citecolor=blue}
   \usepackage{graphicx}
     \usepackage[margin=1.5in]{geometry}
   \usepackage[font=scriptsize]{caption}
\usepackage{todonotes}
\usepackage{tikz}
\usetikzlibrary{shapes.geometric,decorations.pathreplacing}
\usetikzlibrary{positioning}
\CDat
\usepackage{scalerel}[2016/12/29]

 \subjclass{Primary: 55Q25, 55P91, 55P92, 55S91.}
\newtheorem{thm}{Theorem}[section]  
\newtheorem*{un-no-thm}{Theorem}
\newtheorem{lem}[thm]{Lemma}         

\newtheorem{bigthm}{Theorem}

\newtheorem{bigcor}[bigthm]{Corollary}

\newtheorem{bigadd}[bigthm]{Addendum}

\theoremstyle{definition}
\newtheorem{defn}[thm]{Definition}   

\theoremstyle{definition}
\newtheorem{notation}[thm]{Notation}

\theoremstyle{definition}

\theoremstyle{definition}

\theoremstyle{remark}
\newtheorem{rem}[thm]{Remark}

\newtheorem*{acks}{Acknowledgements}

\newtheorem*{out}{Outline}

\newtheorem{ex}[thm]{Example}

\DeclareMathOperator{\Sp}{Spectra}

\DeclareMathOperator{\tr}{tr}

\DeclareMathOperator{\map}{map}
\DeclareMathOperator*{\colim}{colim}

\DeclareMathOperator*{\nat}{nat}

\begin{document}
\title{On the stable Hopf invariant}
\date{\today}
\author{John R.\ Klein}
\address{Wayne State University, Detroit, MI 48202}
\email{klein@math.wayne.edu}
\begin{abstract}  We present a simplified approach to the
 stable Hopf invariant. We provide short elementary proofs
of the Cartan Formula, the Composition Formula, and
the Transfer formula. 
In addition, 
when $\pi$ is a discrete group, we show how to extend these results 
to the stable category of $\pi$-spaces. We also consider the extent to which the stable Hopf invariant is unique.
\end{abstract}
\maketitle
\setlength{\parindent}{15pt}
\setlength{\parskip}{1pt plus 0pt minus 1pt}

\def\Sp{\text{\bf Sp}}
\def\vo{\varOmega}
\def\vs{\varSigma}
\def\smsh{\wedge}
\def\flush{\flushpar}
\def\id{\text{id}}
\def\dbslash{/\!\! /}
\def\codim{\text{\rm codim\,}}
\def\:{\colon}
\def\holim{\text{holim\,}}
\def\hocolim{\text{hocolim\,}}
\def\Bbb{\mathbb}
\def\bold{\mathbf}
\def\Aut{\text{\rm Aut}}
\def\cal{\mathcal}
\def\sec{\text{\rm sec}}
\def\gda{G\text{\rm -}\delta\text{\rm -}\alpha}
\def\PDD{\text{\rm pd\,}}
\def\PD{\text{\rm P}}
\def\stableto {\,\, \mapstochar \!\!\to}

\setcounter{tocdepth}{1}
\tableofcontents
\addcontentsline{file}{sec_unit}{entry}


\section{Introduction} \label{sec:intro}
Heinz Hopf introduced what is now called the Hopf invariant in his 1931 work on maps between spheres as a way to distinguish homotopy classes of maps that are invisible to homology \cite{Hopf}. In the 1950s, this notion was generalized by James \cite[pp.~192–193]{James-suspension-triad}, who defined invariants associated to desuspending maps of based spaces $\Sigma A \to \Sigma B$. For each positive integer $n$, the $n$th James–Hopf invariant is an operation on homotopy classes
\[
\gamma_n \colon [\Sigma A,\Sigma B] \longrightarrow [\Sigma A,\Sigma B^{[n]}]
\]
which is natural in both $A$ and $B$, where $B^{[n]}$ denotes the $n$-fold smash product of $B$ with itself.

Boardman and Steer \cite{Boardman-Steer} instead consider the suspended invariants
\[
\lambda_n = \Sigma^{n-1}\gamma_n \colon [\Sigma A,\Sigma B] \longrightarrow [\Sigma^n A,\Sigma^n B^{[n]}]
\]
and show that they are axiomatically characterized by certain properties encoded in what they call a \emph{Hopf ladder}. A Hopf ladder is a collection of such natural transformations satisfying $\lambda_1 = \mathrm{id}$, $\lambda_n(\Sigma f) = 0$ for $n \ge 2$, and the Cartan formula
\[
\lambda_n (f + g) = \sum_{i+j=n} \lambda_i(f)\cdot \lambda_j(g)\, .
\]

In the 1970s,  the  {\it Segal-Snaith Hopf invariants}  $h_n$ were introduced using the approximation theorem in combination with the Snaith splitting  
\cite{Snaith-stable}, \cite{May_approximation}, \cite{Cohen-May-Taylor}, \cite{Cohen_snaith}, \cite{Segal1973}.
Consider the infinite loop space $Q(B) = \Omega^\infty \Sigma^\infty  B$ for a based space $B$. The Snaith
 splitting provides a decomposition of the suspension spectrum
\[
\Sigma^\infty Q(B) \simeq \bigvee_{n \ge 0} \Sigma^\infty D_n(B)\, ,
\]
where $D_n(B) := B^{[n]}\smsh_{\Sigma_n} ({E\Sigma_n}_+)$ is the $n$-{\it adic construction}, i.e. the reduced homotopy orbits
of the symmetric group $\Sigma_n$ acting on the $n$-fold smash product $B^{[n]}$.
We will use of a version of the Snaith splitting provided by Goodwillie's calculus of functors \cite{Goodwillie_calc3} which has the advantage of
being natural.\footnote{The argument of \cite[ex.~1.20]{Goodwillie_calc3}, which is only stated for $\Omega \Sigma X$, also applies to $Q(X)$.}

If we compose the Snaith splitting  with projection onto the 
$n$th summand and take the adjoint,  we obtain a map
\[
Q(B) \to Q(D_n(B))\, .
\]
The latter induces the operation
\[
h_n\: \{A,B\} \to \{A,D_n(B)\}\, 
\]
by applying homotopy classes $[A,{-}]$. In the display, $\{A,B\} = [A,Q(B)]$ is the abelian of homotopy classes of stable maps $A \to B$.
As far as I am aware, a Hopf-ladder type axiomatic characterization of the invariants $h_n$ is currently unknown.

In this paper we restrict our attention to the case $n=2$. In this case there is a different construction of an operation 
\[
H\: \{A,B\} \to \{A,D_2(B)\}
\]
which avoids the Snaith splitting (I first learned
of such a construction from Andrew Ranicki).  Our main goal is to develop and analyze the fundamental properties of $H$.

Let
\[
E\: [A,B] \to \{A,B\}
\] be the {\it stabilization map}, i.e., the passage from unstable homotopy classes to stable homotopy classes. 
If $f,g\in \{A,B\}$ then the {\it cup product} \cite{Boardman-Steer}
\[
f\cup g\in \{A,B \smsh B\}
\]
is the stable composition
\[
A @> \Delta_A >> A \smsh A @> f \smsh g >> B \smsh B\, .
\]
The {\it symmetrized cup product} 
\[
f\cup_2 g \in \{A,D_2(B)\}
\]
is the composition 
\[
A @> f\cup g >> B\smsh B \to D_2(B)
\]
where the map $B\smsh B \to D_2(B)$ is the homotopy quotient by the cyclic group $\Bbb Z_2= \Sigma_2$.

We are now in position to state the main result.

\begin{bigthm} \label{bigthm:Hopf-properties} There is an operation
\[
H\: \{A,B\} \to \{A,D_2(B)\}
\]
called the  {\it stable Hopf invariant} that
is natural in both $A$ and $B$ and which 
satisfies:
\begin{enumerate}[label=(\roman*).]
\item (Normalization). $H(E(f)) = 0$;
\item (Cartan Formula). $H(f+g) = H(f) + H(g) + f \cup_2 g$; 
\item (Transfer Formula). $\tr H(f)=f\cup f - \Delta_B \circ f$;
\item (Composition Formula). $H(g\circ f) = H(g)\circ f + D_2(g) \circ H(f)$.
\end{enumerate}
\end{bigthm}

\begin{rem}  The transfer homomorphism $\tr\: \{A,D_2(B)\} \to \{A , B\smsh B\}$ arises from the regular $\Bbb Z_2$-covering space pair
\[
((B\times B) \times E\Bbb Z_2, (B \vee B) \times E\Bbb Z_2) \to ((B\times B) \times_{\Bbb Z_2} E\Bbb Z_2, (B \vee B) \times_{\Bbb Z_2} E\Bbb Z_2)
\] 
\cite[const.~4.1.1]{Adams_infinite}.
For the Segal-Snaith Hopf invariant, the Cartan formula appears in 
\cite{Caruso-Cohen-May-Taylor} (for a geometric description, see \cite[thm.~2.2]{Koschorke-Sanderson}; see also 
\cite{Kuhn_James-Hopf} for additional properties of the Segal-Snaith Hopf invariant).
For the {\it geometric Hopf invariant} of Crabb and Ranicki, a version of the Transfer and Composition Formulas appear in \cite[prop.~5.33]{Crabb-Ranicki}. 

Our approach most closely resembles the the one taken by Crabb and Ranicki with some key differences:
They fix a $\Bbb Z_2$-representation $V$ and define the geometric Hopf invariant as a natural transformation
\[
h_V \: [S^V \smsh A , S^V \smsh B]_{\Bbb Z_2} \to [\Sigma S(\alpha \otimes V)_+ \smsh S^V \smsh A , S^{(\alpha +1) \otimes V} \smsh B]_{\Bbb Z_2}\, ,
\]
where the domain is the set of homotopy classes of $\Bbb Z_2$-equivariant maps $S^V \smsh A \to S^V \smsh B$. In the display, $S^V$ denotes the one-point compactification of $V$, 
$\alpha$ is the sign representation, and $S(\alpha \otimes V)$ is the unit sphere.

By constrast, the source and target of the operation $H$  do not involve a choice of representation and 
are expressed in terms of {\it unequivariant} stable homotopy classes. Even though $H$ is phrased in unequivariant language, its
actual construction requires passing through the $\Bbb Z_2$-equivariant stable category.  
The operations $H$ and $h_V$ are related by stabilization by representations and the application of a transfer map.
\end{rem}


\subsection{Relation to the Segal-Snaith operation}
Let  $h := h_2$ be the second Segal-Snaith Hopf invariant.

\begin{bigthm} \label{bigthm:coincide}  The operations $h,H\:  \{A,B\} \to \{A,D_2(B)\}$ coincide.
\end{bigthm}

Suppose that $A$ is a CW complex of dimension $s$ and $Y$ is $N$-connected. Then
 {\it stable range} occurs when $s \le 2r+1$; in this case $E\: [A,B] \to \{A,B\}$ is surjective by the Freudenthal suspension theorem.
  The {\it metastable range} occurs when $s\le 3r+1$. It is well-known that the Segal-Snaith Hopf invariant
  is the complete obstruction to destabilization in the metastable range \cite{Milgram}. 

\begin{bigcor} \label{bigcor:coincide}  Assume $s\le 3r+1$. Let $f\in \{A,B\}$ be a homotopy class.
Then  $H(f) = 0$ if and only if
$f = E(f')$ for some $f'\in [X,Y]$.
\end{bigcor}

A natural transformation is said to possess the {\it EHP property} if it satisfies the conclusion of Corollary \ref{bigcor:coincide}.

\subsection{Extension to $\pi$-spaces}
For applications to surgery theory, it is useful to generalize the above to 
the $\pi$-equivariant setting, where $\pi$ is any discrete group. 
Let $A$ and $B$ be based $\pi$-CW complexes such that the action is free away from the basepoint.
In this instance, we have an {\it equivariant stable Hopf invariant}
\[
H^\pi\: \{A,B\}_\pi \to \{A,D_2(B)\}_\pi
\]
where $\{A,B\}_\pi := [A,Q(B)]_{\pi}$ denotes the abelian group of homotopy classes of stable equivariant maps $A\to B$.

\begin{bigadd} \label{bigadd:Hopf-properties}  The equivariant stable Hopf invariant $H^\pi$ possesses  properties (i)-(iv).
Furthermore, the $\pi$-equivariant analog of Corollary \ref{bigcor:coincide} holds: In the metastable range, $f\in \{A,B\}_{\pi}$ destabilizes
to $[A,B]_\pi$ if and only if $H^\pi(f)$ is trivial.
\end{bigadd}

\subsection{Concerning uniqueness} Suppose that $\lambda\: \{A,B\} \to \{A,D_2(B)\}$ is a natural transformation
possessing the following properties:
\begin{enumerate}[label=(\alph*).]
\item $\lambda(f) = 0$ if $f$ is represented by an unstable map,
\item $\lambda$ possesses the Cartan formula, and
\item $\lambda$ possesses the EHP property.
\end{enumerate}

Let $A(\Bbb Z_2)$ denote the Burnside ring of the group $\Bbb Z_2$ (cf.~\cite{Bouc_Burnside-survey}).
There is an isomorphism
\[
A(\Bbb Z_2) \cong \Bbb Z[\rho]/(\rho^2-2\rho)\, ,
\]
where $\rho$ corresponds to $\Bbb Z_2$ considered as a $\Bbb Z_2$-set. Let $\hat A(\Bbb Z_2)$
denote the completion of $A(\Bbb Z_2)$ with respect to its augmentation ideal. 
In \S\ref{sec:unique} we show that
$\hat A(\Bbb Z_2)$ acts on the homotopy classes of natural transformations $D_2({-}) \to D_2({-})$.

Let  $\hat A(\Bbb Z_2)^{\times}$ denote the group of units of $\hat A(\Bbb Z_2)$, and let 
$K$ denote the kernel of the split surjective homomorphism $\hat A(\Bbb Z_2)^{\times} \to \hat A(e)^{\times} = \{\pm 1\}$ induced by augmentation.
Then
\[
K \cong \Bbb Z_2\times \Bbb Z^{\smsh}_2, 
\]
where
$\Bbb Z^{\smsh}_2$ is the group of $2$-adic integers. The  $\Bbb Z_2$-factor is generated by $\rho-1$ and the factor
 $\Bbb Z^{\smsh}_2$ is topologically generated by $2\rho-3$.

\begin{bigthm}\label{bigthm:unique} The set of natural transformations $\lambda\: \{A,B\} \to \{A,D_2(B)\}$ which satisfy
(a)-(c) above is a free and transitive $K$-set, i.e., it is a $K$-torsor.
\end{bigthm}

\begin{out} Section \ref{sec:equivariant} introduces $\Bbb Z_2$-equivariant stable homotopy from a low-tech point-of-view. Section \ref{sec:tomDieck}
recalls the tom~Dieck splitting in the case of the group $\Bbb Z_2$. In  Section \ref{sec:stable-Hopf} we provide a construction of the stable Hopf invariant.
Section \ref{sec:main-results} concerns the proof of Theorem \ref{bigthm:Hopf-properties}.
The proof of Theorem \ref{bigthm:coincide} appears in  Section \ref{sec:uniqueness}. Section \ref{sec:equivariant-properties} contains the proof of  Addendum \ref{bigadd:Hopf-properties}.
In Section \ref{sec:unique} we prove Theorem \ref{bigthm:unique}.
\end{out}

\begin{acks} Andrew Ranicki explained to me the construction of $H$ about 25 years ago (a chain level version appears in the proof of
\cite[prop.~1.5]{Ranicki}). Greg Arone was also aware of such a construction.
I am grateful to the referee for a careful reading of the paper, especially for pointing out various errors I made in an ealier draft, 
as well as the observation that stable Hopf invariants are not necessarily unique.
\end{acks}

\section{$\Bbb Z_2$-equivariant stable homotopy} \label{sec:equivariant} 
Let $T$ be the category of compactly generated weak Hausdorff spaces
and let $T_\ast$ be the corresponding category of based spaces.
We let $T_\ast(\Bbb Z_2)$ be the category of based spaces with (left)  $\Bbb Z_2$-action.
If $X$ is an unbased $\Bbb Z_2$-space, we let $X_+$ denote the corresponding
based one obtained from $X$ by taking a disjoint basepoint.

For based $\Bbb Z_2$-spaces $A,B$,  the function space of unequivariant based maps
\[
\map_\ast(A,B)
\]
is equipped with a preferred $\Bbb Z_2$-action, where the action is defined by conjugating functions.

Up to isomorphism, there are exactly two irreducible
real orthogonal representations of $\Bbb Z_2$ of positive dimension. These are the trivial representation $1$ and the sign representation $\alpha$.
We choose a {\it complete universe} $\cal U$ for $\Bbb Z_2$ given by the direct sum of countably many copies $1$ and $\alpha$.
Then every finite dimensional subrepresentation of $\cal U$ is a direct sum of the form
\[
s\alpha + t \, ,
\]
for non-negative integers $s$ and $t$.

If $V\subset \cal U$ is a finite dimensional subrepresentation, then we let $S^V$ denote its one-point compactification. The latter is a sphere
of dimension $\dim V$ equipped with a based $\Bbb Z_2$-action. 

Suppose that $Y$ is a based $\Bbb Z_2$-CW complex. We give $S^V\smsh Y$ the diagonal action. The based function space
\[
\Omega^V\Sigma^V (Y) := \map_\ast(S^V,S^V \smsh Y)
\]
is equipped with a  $\Bbb Z_2$-action. Then the colimit
\[
Q_{\Bbb Z_2}(Y) := \colim_{V\in \cal U} \Omega^V\Sigma^V (Y)
\]
is equipped with a $\Bbb Z_2$-action. The latter is the zeroth space of a $\Bbb Z_2$-equivariant spectrum
\[
\Sigma^\infty_{\Bbb Z_2}(Y)
\]
whose $V$-th space is $Q_{\Bbb Z_2}(S^V\smsh Y)$.

\begin{defn} Let $X$ be a based $\Bbb Z_2$-CW complex. Set
\[
\{X,Y\}_{\underline{\Bbb Z_2}} := \colim_{V\subset \cal U} \pi_0(\map_\ast(S^V \smsh X,S^V \smsh Y)^{\Bbb Z_2}) \, .
\]
\end{defn}

\begin{rem} Alternatively, 
\[
\{X,Y\}_{\underline{\Bbb Z_2}}  =  \pi_0(\map_\ast(X,Q_{\Bbb Z_2}(Y))^{\Bbb Z_2})
\]
where $\map_\ast(X,Q_{\Bbb Z_2}(Y))^{\Bbb Z_2}$ is the based mapping space of $C$-equivariant maps
$X\to Q_{\Bbb Z_2}(Y)$.
\end{rem}

\begin{defn}[$\Bbb Z_2$-equivariant Stable Maps]
Suppose that $X$ and $Y$ are  based $\Bbb Z_2$-CW complexes.  An {\it equivariant stable map} $X \to Y$
is an equivariant map $X\to Q_{\Bbb Z_2}(Y)$. 
\end{defn}

Equivariant stable maps can be composed and form an $\infty$-category. 
If $X$ is finite, then there is a finite dimensional representation $V\subset \cal U$
such that a stable map is represented by an equivariant map $S^V \smsh X \to S^V \smsh Y$, where two such maps agree if they
coincide in the colimit of the mapping space.

\section{The tom~Dieck splitting for $\Bbb Z_2$} \label{sec:tomDieck}
Let $X$ be a free unbased $\Bbb Z_2$-CW complex of finite type with orbit space  $X/\Bbb Z_2$.
Then one has an {\it Adams isomorphism} \cite[\S II.7]{May_equivariant}
\begin{equation} \label{eqn:Adams-iso}
\Sigma^\infty (X/\Bbb Z_2)_+ @> \simeq >> (\Sigma^\infty_{\Bbb Z_2}X_+)^{\Bbb Z_2}\, ,
\end{equation}
where the target is the genuine fixed points of $\Bbb Z_2$ acting on the genuine $\Bbb Z_2$-spectrum $\Sigma_{\Bbb Z_2}(X_+)$.
Note that $X/\Bbb Z_2$ is equipped with the trivial action of $\Bbb Z_2$.

The map of spectra \eqref{eqn:Adams-iso} represents an element of the equivariant stable homotopy group
\[
 \{(X/\Bbb Z_2)_+, X_+\}_{\underline{\Bbb Z_2}} 
 \]
and is
induced by the stable transfer map for the regular covering space
\[
p\: X \to X/\Bbb Z_2 \, .
\]
\begin{ex} 
Here is a sketch of the homotopy class of the Adams isomorphism in a special case: Assume that $X$ is a smooth  
compact codimension zero $\Bbb Z_2$-submanifold of  $V := \Bbb R^{n\alpha}$ for $n$ sufficiently large.
Let $e\: X\to V$ be the inclusion.
Then the embedding $(e,p)\: X \to V\times X/{\Bbb Z_2}$ has normal bundle the trivial $G$-bundle $X\times V \to X$.   
The Pontryagin-Thom construction of $(e,p)$  is therefore a $\Bbb Z_2$-equivariant map
\[
p^!\: S^V \smsh (X/\Bbb Z_2)_+ \to S^V \smsh X_+\, 
\] 
which defines the desired element of $\{(X/\Bbb Z_2)_+, X_+\}_{\underline{\Bbb Z_2}}$.
\end{ex}

\begin{ex} \label{ex:x-plus-y} Let $X = \{x_-,x_+\} \subset V$ be an equivariant subset of cardinality two, i.e., $x_- = -x_+$. 
In this case, the Adams isomorphism is an equivalence
\[
S^0 @>\simeq >> (\Sigma^\infty_{\Bbb Z_2} X_+)^{\Bbb Z_2}
\]
whose associated homotopy class may be described as follows:
Choose an equivariant  tubular neighborhood $U = U_{x_-} \amalg U_{x_+}$
of $X$. Then the interior of $U$ is identified with $V \amalg V$. Apply the Pontryagin-Thom construction to obtain an equivariant map
\[
x_-+_t x_+\: S^V  \to S^V \smsh X_+ \, ,
\]
where in the above the expression $+_t$ indicates the dependence  on the choice of  tubular neighborhood. The equivariant homotopy class
of $x_-+_t x_+$ is the desired  element of $\{S^0,X_+\}_{\Bbb Z_2}$. 
whose associated {\it un}equivariant homotopy class represents the sum $x_- + x_+\in \{S^0,X_+\}$, where 
$x_-,x_+\: S^0 \to X_+$ now denote the evident inclusions. 

However, the reader should not make the mistake of
thinking of $x_- +_t x_+$ as the result of summing $x$ and $y$, since the latter do not represent elements of $\{S^0,X_+\}_{\underline{\Bbb Z_2}}$.
\end{ex}

\begin{ex} \label{ex:x-plus-y-again}  The previous example indicates how to obtain an explicit 
description of the Adams isomorphism in the general case. Let $E_n(2)$ be the second space of the little $n$-cubes operad. A point of $E_n(2)$
is a rectilinear embedding
\[
(0,1)^{\times n}\times  \{1,2\} \to (0,1)^{\times n}\, .
\]
Then $E_n(2)$ is a free $\Bbb Z_2$-space having the equivariant homotopy type of $S^{n-1}$ with its antipodal action.
The quotient map $E_n(2) \to E_n(2)/\Bbb Z_2$ is a regular two-fold cover.

Let $X$ be a free $\Bbb Z_2$-CW complex which for simplicity we take to be finite. 
If $n$ is sufficiently large with respect to the dimension of $X$, then the two fold cover $X\to X/\Bbb Z_2$ is classified by a map
\[
t\: X/\Bbb Z_2 \to E_n(2)/\Bbb Z_2
\]
in the sense that there is a pullback square
\[
\xymatrix{
X \ar[r]^{\tilde t}\ar[d] & E_n(2) \ar[d] \\
X/\Bbb Z_2 \ar[r]_(.4){t} & E_n(2)/\Bbb Z_2\, .
}
\]
where the top horizontal map $\tilde t$ is equivariant. 

If $[x]\in X/\Bbb Z_2$ is a point, then let $\{x_-,x_+\}$ denote its preimage in $X$. 
Then $t(x) \in E_n(2)/\Bbb Z_2$  defines a tubular neighborhood of the center of
each little cube in the configuration. Identify $x_\pm$ with the center of the first cube in $\tilde t(x_\pm)$. Then Pontryagin-Thom construction 
defines an equivariant map
\[
x_- +_t x_+\: S^V \to S^V \smsh \{x_-,x_+\}_+ @> \subset >> S^V \smsh X_+\, ,
\]
where $V = \Bbb R^{n\alpha}$ is $\Bbb R^n$ with the antipodal action. We emphasize here that 
the notation  $+_t$ indicates the dependence on the little $n$-cube configuration
$t(x) \in E_n(2)/\Bbb Z_2$. 

The operation $[x] \mapsto x_- +_t x_+$ 
defines a map
\[
(X/\Bbb Z_2)_+ \to (\Omega^V(S^V \smsh X_+))^{\Bbb Z/2} 
\]
which passes to the Adams isomorphism upon stabilization.
\end{ex}

\subsection{The relative case}
More generally, suppose that $Y$ is an arbitrary based $\Bbb Z_2$-CW complex.
If we replace $Y$ by $Y \times E\Bbb Z_2$, then one has a regular covering space pair
\[
(Y \times E\Bbb Z_2, E\Bbb Z_2) \to (Y \times_{\Bbb Z_2} E\Bbb Z_2,B\Bbb Z_2)
\]
and therefore a relative transfer map of pairs which induces on quotients
an equivalence
\[
\Sigma^\infty(Y_{h\Bbb Z_2}) @>\simeq >> \Sigma^\infty_{\Bbb Z_2} (Y \smsh {E\Bbb Z_2}_+)^{\Bbb Z_2}\, , 
\]
where $Y_{h\Bbb Z_2} := Y \smsh_{\Bbb Z_2} {E\Bbb Z_2}_+$ is the reduced Borel construction of $\Bbb Z_2$ acting on $Y$.

\begin{rem}  The classical transfer map $\tr\: \Sigma^\infty Y_{h\Bbb Z_2} \to  \Sigma^\infty Y$
of the above covering space pair of 
 is  the composition
\[
\Sigma^\infty Y_{h\Bbb Z_2}  \to (\Sigma^\infty_{\Bbb Z_2} Y)^{\Bbb Z_2} @> \subset >> \Sigma^\infty_{\Bbb Z_2}(Y) \simeq \Sigma^\infty (Y)\, .
\]
\end{rem}

With $Y$ as above, there is always a split cofiber sequence of spectra
\[
\Sigma^\infty Y_{h\Bbb Z_2} \to (\Sigma^\infty_{\Bbb Z_2}Y)^{\Bbb Z_2} @>\eta >> \Sigma^\infty (Y^{\Bbb Z_2})\, .
\]
The first map in the sequence  is a composition:
\[
\Sigma^\infty Y_{h\Bbb Z_2}  @> \simeq >>  (\Sigma^\infty_{\Bbb Z_2}(Y \smsh {E\Bbb Z_2}_+))^{\Bbb Z_2} \to  (\Sigma^\infty_{\Bbb Z_2} Y)^{\Bbb Z_2}\, ,
\]
where the first map is the Adams isomorphism and the second map is
induced by collapsing $E\Bbb Z_2$ to a point.

The splitting map $\Sigma^\infty (Y^{\Bbb Z_2}) \to (\Sigma^\infty_{\Bbb Z_2}Y)^{\Bbb Z_2}$ is induced by the inclusion $Y^{\Bbb Z_2} \to Y$, more precisely, it is adjoint to the map
\[
Y^{\Bbb Z_2} \to (Q_{\Bbb Z_2}Y)^{\Bbb Z_2} 
\]
which sends a a fixed point of $y\in Y^{\Bbb Z_2}$ to the point represented by the equivariant map $S^0 = S^0 \smsh \{y\}_+ \to S^0 \smsh Y$.

\subsection{Specialization}
Let $B$ be a based CW complex. We will be interested in applying the above observation to the based $\Bbb Z_2$-CW complex
\[
Y = B \smsh B\, ,
\]
where the action of $\Bbb Z_2$ is given by permuting the factors of the smash product. In this case the split cofiber sequence is
\[
\Sigma^\infty D_2(B)  @>\iota_B >>  (\Sigma^\infty_{\Bbb Z_2} (B\smsh B))^{\Bbb Z_2} \to \Sigma^\infty B \, ,
\]
since $(B\smsh B)^{\Bbb Z_2} = B$.
The splitting map for the cofiber sequence  is induced by the reduced diagonal $\Delta_B\: B\to B\smsh B$.


Consider the commutative square
\[
\xymatrix{
(B^{\times 2} \times \Bbb Z_2, B^{\vee 2} \times \Bbb Z_2)  \ar[r] \ar[d] & (B^{\times 2} \times_{\Bbb Z_2} \Bbb Z_2 , B^{\vee 2}, \times_{\Bbb Z_2} \Bbb Z_2)  \ar[d] \\
(B^{\times 2} \times E\Bbb Z_2, B^{\vee 2} \times E\Bbb Z_2)   \ar[r] & (B^{\times 2} \times_{\Bbb Z_2} E\Bbb Z_2 , B^{\vee 2}, \times_{\Bbb Z_2} E\Bbb Z_2)
}
\]
in which the vertical maps are $2$-fold coverings and the horizontal maps are induced by  the equivariant inclusion $\Bbb Z_2 \to E\Bbb Z_2$ (i.e., inclusion $S^0 \subset S^\infty$).

Using naturality, we infer that there is a homotopy commutative diagram
\[
\xymatrix{
\Sigma^\infty(B \smsh B) \ar[r]^(.35){\simeq} \ar[d]_{({-})_2} &  (\Sigma^\infty_{\Bbb Z_2} (B\smsh B \smsh {\Bbb Z_2}_+))^{\Bbb Z_2}\ar[d]   \ar[r] & \ast \ar[d] \\
 \Sigma^\infty (D_2(B))  \ar[r]_{\iota_B}  &  (\Sigma^\infty_{\Bbb Z_2} (B\smsh B))^{\Bbb Z_2}  \ar[r]  & \Sigma^\infty B\\
 }
 \]
 whose rows are the tom Dieck cofiber sequences.
The  map $({-})_2$ is induced by the inclusion $\Bbb Z_2 \to E\Bbb Z_2$:
\[
B \smsh B = (B \smsh B)\smsh_{\Bbb Z_2} {\Bbb Z_2}_+@> >> (B \smsh B)\smsh_{\Bbb Z_2} {E\Bbb Z_2}_+ = D_2(B)\, .
\]

\section{Definition of the stable Hopf invariant} \label{sec:stable-Hopf}
Given a stable map $f\: A \to B$, i.e., a map $A\to Q(B)$, we may associate two equivariant stable maps
$\Delta_B \circ f, (f\smsh f)\circ \Delta_A\: A\to B\smsh B$. 

Without loss in generality, we may assume that $A$ is a finite complex; then
$f\: S^n \smsh A \to S^n \smsh B$.
The map  $ \Delta_B\circ f$ is given by the composition
\[
S^n \smsh A @>f >> S^n \smsh B @> \Delta_B >> S^n \smsh B \smsh B\, .
\]
For the second map, take $f\smsh f$ and shuffle
To obtain an equivariant map
\[
S^n \smsh S^n \smsh A \smsh A \cong S^n \smsh A \smsh S^n \smsh A @>f \smsh f >>  S^n \smsh B \smsh S^n \smsh B \cong S^n \smsh S^n \smsh B \smsh B \, .
\]
where $S^n \smsh S^n$ is given the action that switches factors. Next, note that 
\[
S^n \smsh S^n \cong S^{n\alpha + n}\, ,
\]
where $\alpha$ is the sign representation,
and $n\alpha + n$ is notation for the direct sum of $n$-copies of $\alpha$ with $n$-copies of the trivial representation.

Making this identification and taking the composition
\[
S^{n\alpha + n} \smsh A @>\Delta_A >> S^{n\alpha + n} \smsh A \smsh A @> f \smsh f >> S^{n\alpha + n} \smsh B \smsh B 
\]
yields the desired equivariant stable map $(f\smsh f)\circ \Delta_B$.

By taking adjunctions, we have constructed two maps $A\to Q_{\Bbb Z_2}(B\smsh B)^{\Bbb Z_2}$. Moreover, it is readily checked that the diagram 
\[
A \, \raisebox{-0.6ex}{\stackon[-2pt]{$@>>\phantom{aa}\Delta_B \circ f\phantom{aa}>$}{$@>(f\smsh f)\circ \Delta_A>>$}}\,  Q_{\Bbb Z_2}(B\smsh B)^{\Bbb Z_2}
@> \eta >> Q(B^{\Bbb Z_2})
\]
commutes, i.e., $\eta $ coequalizes the maps $\Delta_B \circ f, (f\smsh f)\circ \Delta_A$.
It follows that the homotopy class of the difference
\[
(f\smsh f)\circ \Delta_A - \Delta_B \circ f \in \{A,B\smsh B\}_{\underline{\Bbb Z_2}}
\]
is in the image of the split injection
\[
\iota_B\: \{A,D_2(B)\} \to \{A,B\smsh B\}_{\underline{\Bbb Z_2}}\, .
\]
\begin{defn} \label{defn:defining-identity}
The {\it stable Hopf invariant} is the natural transformation
\[
H\: \{A,B\} \to \{A,D_2(B)\}
\]
such that $H(f) \in \{A,D_2(B)\} $ is the unique element such that
\begin{equation} \label{eqn:defining-identity}
\iota_B H(f) = (f\smsh f)\circ \Delta_A - \Delta_B \circ f\, .
\end{equation}
\end{defn}

\begin{rem} \label{rem:defining-identity} The natural transformation $\iota_B \circ H$ is induced by a map of spaces
\[
Q(B) \to Q_{\Bbb Z_2}(B\smsh B)^{\Bbb Z_2}
\]
which may be described as follows: Let $\gamma\: S^n \to S^n \smsh B$ represent a point of $Q(B)$.
Then $\Delta_B \circ \gamma \: S^n \to S^n \smsh B\smsh B$ represents a point of $Q_{\Bbb Z_2}(B\smsh B)^{\Bbb Z_2}$.
Likewise, so does $\gamma \smsh \gamma\: S^n \smsh S^n \to S^n \smsh B \smsh S^n \smsh B = S^n\smsh S^n \smsh B \smsh B$.
For a choice of loop multiplication on $Q_{\Bbb Z_2}(B\smsh B)^{\Bbb Z_2}$, the operation
\[
\gamma\mapsto \gamma \smsh \gamma - \Delta_B \circ \gamma 
\]
induces $\iota_B \circ H$.
\end{rem}

\section{Proof of Theorem \ref{bigthm:Hopf-properties}} \label{sec:main-results}

\subsection{The cup product again} Let $A$ and $B$ be CW complexes.
For stable maps $f,g\: A\to B$, the {\it  cup product }
\begin{equation} \label{eqn:cup}
f\cup g \in \{A,B\smsh B\}
\end{equation}
is defined by the composition of the equivariant stable maps $f\smsh g$ and $\Delta_A$, i.e.,
\[
f\cup g := (f \smsh g) \circ \Delta_A\, .
\]
If $A$ is finite, then $f\cup g$ is given as follows: Let $f\: \Sigma^j A \to \Sigma^j B$ and $g\: \Sigma^j A\to \Sigma^j B$ 
be representatives.
Then $f\cup g$ is given by the composition 
\[
\Sigma^{2j} A @> \Sigma^{2j} \Delta_{A} >> \Sigma^{2j} (A \smsh A)= \Sigma^{j} A \smsh \Sigma^j A @>f \smsh g >> \Sigma^j B \smsh \Sigma^j B = \Sigma^{2j} (B\smsh B)\, .
\]
The  cup product  induces a pairing
\[
\cup \: \{A,B\}\times \{A,B\} \to \{A,B\smsh B\}\, .
\]
The proof of the following result is straightforward and we leave its details to the reader.

\begin{lem} \label{lem:cup-prop} The  cup product  satisfies the following identities:
\begin{enumerate} [label=(\roman*).]
\item $g\cup f = \tau\circ (f\cup g)$, where $\tau\: B\smsh B \to B\smsh B$ is the twist map;
\item $(f+f')\cup g = f\cup g + f'\cup g$.
\end{enumerate}
\end{lem}

\begin{defn} \label{defn:symm-cup}
For equivariant stable maps $f,g\: A\to B$, the {\it symmetrized cup product}
\[
f\cup_2 g := (f\cup g)_2 \in \{A,D_2(B)\}
\] 
is the stable composition
\[
A  @> f\cup g >> B \smsh B @>({-})_2 >> D_2(B)\, .
\]
\end{defn}

Let  $\tau \:  B\smsh B \to B\smsh B$ the map which permutes factors.
Consider the map
\[
\iota_B\: \Sigma^\infty D_2(B) \to (\Sigma^{\infty}_{\Bbb Z_2} (B\smsh B))^{\Bbb Z_2}
\]
The for $f,g\in \{A,B\}$ we may form 
\[
\iota_B (f\cup_2 g) \in \{A,B\smsh B\}_{\underline{\Bbb Z_2}}
\]
\begin{notation} \label{notation:transfer2} For $f,g\in \{A,B\}$ we set
\[
\iota_B \circ (f\cup_2 g)  :=  f\cup g +_t  g\cup f \, ,
\]
cf.~Example \ref{ex:x-plus-y-again}.
\end{notation}

\begin{rem} \label{rem:transfer2}  The elements $f\cup g,g\cup f\in \{A,B\smsh B\}$ do not 
make sense as an elements 
$\{A,B\smsh B\}_{\underline{\Bbb Z_2}}$.
However, it is an elementary exercise involving the transfer to show that forgetful homomorphism
\[
 \{A,B\smsh B\}_{\underline{\Bbb Z_2}} \to  \{A,B\smsh B\}
\]
sends the element $f\cup g +_t  g\cup f$ to the element
\[
f \cup g + g\cup f = (1+\tau)(f\cup g)\, .
\]
This observation partially justifies the notation. Note that symbol $+_t$ exhibits the dependence of the construction
on our configuration space model for the Adams isomorphism.
\end{rem}

\begin{proof}[Proof of Theorem \ref{bigthm:Hopf-properties}]
As remarked above, $\iota_B \: \{A,D_2(B)\} \to \{A,B\smsh B)\}_{\underline{\Bbb Z_2}}$ is a split injection.
Recall that $H(f)$ is defined using the identity  \eqref{eqn:defining-identity}.
For the Normalization property, it suffices to note that when $f\: A\to B$ is an unstable map, then the square
\[
\xymatrix{
A \ar[r]^f \ar[d]_{\Delta_A} & B \ar[d]^{\Delta_B} \\
A\smsh A \ar[r]_{f\smsh f} & B\smsh B
}
\]
commutes. It follows that the difference $\iota_B H(f) :=  (f\smsh f) \circ \Delta_A - \Delta_B \circ f \in \{A,B\smsh B\}_{\underline{\Bbb Z_2}}$ is trivial.

The Transfer Formula is also an easy consequence of the defining identity \eqref{eqn:defining-identity},
 since the composition
 \[
 \{A,D_2(B)\}@> \iota_B >> \{A,B\smsh B)\}_{\underline{\Bbb Z_2}} \to \{A,B\smsh B\}
 \]
gives the transfer, where the second displayed map is the forgetful homomorphism.
 
The Composition Formula also follows from \eqref{eqn:defining-identity}:
If $f\: A\to B$ and $g\: B \to C$, then
 \begin{align*}
 \iota_C H(g\circ f) &= (g\circ f) \smsh (g\circ f) \circ \Delta_A - \Delta_C \circ g\circ f \, , \\
                  &=       (g\smsh g)\circ (f\smsh f) \circ \Delta_A  - (\Delta_C \circ g)\circ f \, , \\
                  &=       (g\smsh g) \circ ((f\smsh f) \circ \Delta_B-   \Delta_B \circ f)      + (g\smsh g) \circ \Delta_B - (\Delta_C \circ g)\circ f\, , \\
                   &=   (g\smsh g)\circ \iota_B H(f)  + \iota_C H(g) \circ f\, , \\
                   & = \iota_C ( D_2(g) \circ H(f)  + H(g) \circ f )\, .
                  \end{align*}
For the Cartan formula, we compute 
{\Small\begin{align*}
\iota_B H(f+g) &= ((f+g) \smsh (f+g)) \circ \Delta_A - \Delta_B\circ (f+g)\, , \\
            &= (f\smsh f) \circ \Delta_A    + ((f\smsh g) \circ \Delta_A    +_t  (g \smsh f) \circ \Delta_A)    + (g\smsh g) \circ \Delta_A)    - (\Delta_B \circ f  + \Delta_B \circ g)\, , \\
            &= ((f\smsh f) \circ \Delta_A - \Delta_B \circ f))  + ((g\smsh g) \circ \Delta_A - \Delta_B \circ g)+  ((f \smsh g) \circ \Delta_A    +_t  (g\smsh f) \circ \Delta_A) \, ,\\
&=  \iota_B H(f)           + \iota_B H(g)+  ( f\cup g +_t g\cup f)  \, ,\\
&=  \iota_B H(f)       + \iota_B H(g) +  \iota_B (f\cup_2 g)     \, ,\\
&=  \iota_B (H(f) +H(g) + f\cup_2 g )\, . \qedhere
\end{align*}}
\end{proof}

\section{Proof of Theorem \ref{bigthm:coincide}} \label{sec:uniqueness}

In this paper we provided an alternative construction of a stable Hopf invariant using the tom~Dieck splitting. 
Here we will verify that  our construction coincides with the Segal-Snaith Hopf invariant.

Recall from  \cite{May_approximation}, \cite{Segal1973} that there is a homotopy functor on based spaces $C({-})$, and a natural transformation
\begin{equation} \label{eqn:eta}
C(X) @> \eta >> Q(X)\, ,
\end{equation}
which is a weak equivalence when $X$ is cofibrant and connected.

The space $C(X)$ is given by a filtration $C_1(X) \subset C_2(X) \subset \cdots$ and 
\[
\mathop{\mathsmaller{\coprod}}_{n\ge 0} (X^{\times n} \times_{\Sigma_n} E(n))/\!\! \sim\, ,
\]
where $E(n) := \lim_{k\to \infty} E_k(n)$ is the space of $n$-disjoint little cubes in $\Bbb R^\infty$,  in which $E_k(n)$ is the space of $n$ little $k$-cubes in $\Bbb R^k$. Note that
$E(n)$ is a model for $E\Sigma_n$, i.e., a free contractible $\Sigma_n$-space, and the $\{E(n)\}_{n \ge 1}$ collectively form
the little $\infty$-cubes operad. In the display,
the  equivalence relation is generated by two kinds of operations: The first operation identifies a representative  $(x,c) \in X^{\times (n-1)} \times  E(n)$
with $(s_j x,c)$, where $s_j$ inserts the basepoint in the $j$-factor of $X^{\times n}$ and the second
operation identifies $(x,c)$ with $(x,d_j(c))$, where $d_j(c): E(n) \to E(n-1)$ 
is given by deleting the $j$-th cube.

The Pontryagin-Thom construction defines an equivariant map \eqref{eqn:eta}.
In particular, if $[x,y,t] \in X^{\times 2} \times_{\Sigma_2} E(2)$, then 
\[
\eta([x,y,t]) = x+_t y = y +_{-t} x \in Q(X) \, .
\]
where $-t$ switches the order of the pair of little cubes defined by $t$.

 Furthermore,
there is a natural equivalence 
\[
\Sigma^\infty C(X) @> \simeq >> \textstyle\bigvee\limits_{n\ge 1} \Sigma^\infty D_n(X)
\]
such that the composition
\[
\Sigma^\infty ((X^{\times n} \times_{\Sigma_n} {E_n})_+) \to \Sigma^\infty C(X) @> \simeq >> \textstyle\bigvee\limits_{n\ge 1}  \Sigma^\infty D_n(X)
\]
is homotopic to the standard  map into the $n$-summand \cite{Cohen_snaith}.

In addition, if we pass to infinite loop spaces, then the diagram
of natural transformations
\[
Q(X) @<\simeq << C(X) @>\simeq >> \textstyle\prod \limits_{n\ge 1}  Q(D_n(X)) @>\text{project} >>   Q(D_2(X)) 
\]
induces the second Segal-Snaith Hopf invariant $h$. Consequently, we may take the above description as a definition of  $h$ (see
also \cite{Segal_hopf} for Segal's approach).

\begin{proof}[Proof of Theorem \ref{bigthm:coincide}] By the above discussion, the operations
$h$ and $H$ are induced by natural transformations of spectrum-valued homotopy functors
\[
{\bold h},{\bold H}\: \textstyle\bigvee\limits_{n\ge 1} \Sigma^\infty D_n(X)  \to \Sigma^\infty D_2(X) \, ,
\]
where ${\bold h}$ is given by projection onto the second summand.

Consequently, for every integer $n \ge 1$, we have a pair if natural transformations of homotopy functors
\[
{\bold h}_n,{\bold H}_n \: \Sigma^\infty D_n(X) \to  \Sigma^\infty D_2(X)\, 
\]
defined respectively by the restrictions of ${\bold h},{\bold H}$ to $\Sigma^\infty D_n(X) $.
By definition, the natural transformation $\bold h_n$ is trivial if $n\ne 2$ and $\bold h_2$ is the identity.

Since the functor $X\mapsto \Sigma^\infty D_n(X)$ is homogeneous of degree $n$, by standard Goodwillie calculus arguments
\cite{Goodwillie_calc3}, the natural transformations  ${\bold H}_n$ are trivial for $n > 2$.
 It therefore suffices to show that ${\bold H}_1$ is homotopically trivial and that ${\bold H}_2$ is homotopic to the identity.

The natural transformation ${\bold H}_1\: \Sigma^\infty X \to \Sigma^\infty D_2(X)$ is trivial since for the identity map $1_X\: X\to X$, we have that
${\bold H_1} = {\bold H}_1 \circ 1_X  = {\bold H} \circ 1_X$ must necessarily vanish, since $1_X$ is unstable.

It remains to prove that ${\bold H}_2 = \text{id}$. This will require some preparation.
 It will be convenient to redefine $D_2(X)$ as $X^{[2]} \smsh_{\Bbb Z_2} E(2)_+$.
It is well known that the evident map of based spaces $X_+ \to X$  admits a stable section up to homotopy, i.e.,
there is a stable map $X\to  X_+$ such that the stable composite
$X\to X_+ \to X$ is homotopic to the identity. This implies that
the induced map $c\: D_2(X_+) \to D_2(X)$ also admits a stable section up to homotopy. Then
for any infinite loop space $Z$, the induced homomorphism $c^\ast\: [D_2(X),Z] \to [D_2(X_+),Z]$ is injective. Furthermore,
\[
D_2(X_+) = (X_+)^{[2]} \smsh_{\Bbb Z_2} {E(2)}_+ \cong (X^{\times 2} \times_{\Bbb Z_2} {E(2)})_+\, .
\]
Consider the composition
\begin{equation} \label{eqn:h2}
D_2(X_+) @> c >> D_2(X) @> \hat {\bold h}_2 >>  Q(D_2(X)) @> \iota_X >> Q_{\Bbb Z_2}(X\smsh X)^{\Bbb Z_2}\, ,
 \end{equation}
 where $\hat {\bold h}_2$ is the inclusion map.

If $(x,y,t) \in X^{\times 2} \times E(2)$ is a point, we write $[x,y,t] \in X^{\times 2} \times_{\Bbb Z_2} E(2)$ 
for the projection of $(x,y,t)$ to its $\Bbb Z_2$-orbit.
Then by Example \ref{ex:x-plus-y-again}, the map \eqref{eqn:h2} is up to homotopy given by
 \begin{equation}\label{eqn:xyyx}
 [x,y,t] \mapsto x\smsh y +_t y\smsh x\, ,
 \end{equation}
where again we write $x,y\: S^0 \to X$ for the based maps defined by the points $x,y\in X$ (cf.~ Examples \ref{ex:x-plus-y}, \ref{ex:x-plus-y-again}).
 
Similarly, using definition of $H$, we claim that the composition
\begin{equation} \label{eqn:H2}
D_2(X_+) @> c >> D_2(X) @>\hat {\bold H}_2 >> Q(D_2(X)) @>\iota_X >> Q_{\Bbb Z_2}(X\smsh X)^{\Bbb Z_2}
\end{equation}
(in which $\hat {\bold H}_2$ is adjoint to ${\bold H}_2$)
 is given by 
\begin{equation} \label{eqn:sum-of-maps}
[x,y,t] \mapsto (x+_t y) \smsh (x+_t y) - ((x\smsh x) + y\smsh y)) \, .
\end{equation}
To see this, note that $[x,y,t]$ maps via $\eta$ to the point of $Q(X)$ defined by the composition
\[
x+ _t y \: S^n @> p_{x,y}>>  S^n \smsh \{x,y\}_+ @> 1_{S^n} \smsh i_{x,y} >> S^n \smsh X\, ,
\]
in which  $p_{x,y}$ is the Pontryagin-Thom construction of the configuration and
 $i_{x,y} \:  \{x,y\}_+ \to X$ is the based map induced by the inclusion of subsets. By Remark \ref{rem:defining-identity} 
it's enough to identify \eqref{eqn:sum-of-maps} with  the expression ``$(x +_t y) \smsh (x +_t y) - \Delta_X \circ (x+_t y)$.'' In other words,
 it suffices to prove that $(1_{S^n} \smsh \Delta_X) \circ (x+_t y) = x\smsh x  +_t  y \smsh y$.
 
Since
\[
\Delta_{X}\circ i_{x,y} := i_{(x,x),(y,y)}\: \{(x,x),(y,y)\}_+ \to X\times X\, ,
\]
we have
\begin{align*}
 (1_{S^n}\smsh   \Delta_X) \circ (x+_t y)                &=(1_{S^n} \smsh \Delta_{X}) \circ  (1_{S^n} \smsh i_{x,y})  \circ p_{x,y}\,, \\
								&=  (1_{S^n} \smsh  i_{(x,x),(y,y)}) \circ p_{x,y}\, ,\\
                                                                      & = x\smsh x  +_t  y \smsh y\, .
                                                                     \end{align*}
Observe that as a homotopy class, the map $[x,y,t] \mapsto x\smsh x  +_t  y \smsh y$ does not
depend on $t$. We have therefore established  the claim that \eqref{eqn:H2} is given by
 \eqref{eqn:sum-of-maps}.

By functoriality of the little cubes operad action with respect to the diagonal, the map $D_2(X_+) \to Q_{\Bbb Z_2}(X\smsh X)^{\Bbb Z_2}$ defined by 
$[(x,y),t] \mapsto (x+_t y) \smsh (x+_t y)$ is homotopic to the sum of three maps
\[
[(x,y),t] \mapsto (x\smsh x) + (x\smsh y +_t y \smsh x) + (y\smsh y) \, ,
\]
with each map indicated by parentheses.
Therefore, \eqref{eqn:sum-of-maps} is homotopic to the signed sum of five maps
\begin{equation} \label{eqn:cancel}
[(x,y),t] \mapsto  (x\smsh x) + (x\smsh y +_t y \smsh x) + (y\smsh y) - (x\smsh x) - (y\smsh y) \, , \\
\end{equation} 
On the level of homotopy classes, we may cancel the diagonal terms appearing in \eqref{eqn:cancel}, since the group
$\{D_2(X_+), X\smsh X\}_{\underline{\Bbb Z_2}}$ is abelian. 
Consequently, the map \eqref{eqn:cancel}  is homotopic to the map \eqref{eqn:xyyx}.
We infer that the
\eqref{eqn:h2} and \eqref{eqn:H2} are homotopic.

As $\iota_X$ is a homotopy retract
it follows that $\hat {\bold H}_2 \circ c$ is homotopic to $\hat {\bold h}_2 \circ c$, i.e., on homotopy
classes $c^\ast(\hat {\bold H}_2) = c^\ast(\hat {\bold h}_2)$. Since $c^\ast$ is injective,
we conclude that ${\bold H}_2$ is homotopic to ${\bold h}_2$.
\end{proof}

\section{Proof of Addendum \ref{bigadd:Hopf-properties}} \label{sec:equivariant-properties}

Let $T$ be the Quillen model category of compactly generated weak Hausdorff spaces \cite{Quillen-homotopical}.
A weak equivalence of $T$ is a weak homotopy equivalence. A fibration is a Serre fibration.
A cofibration is a map which possesses the left lifting property with respect to the acyclic fibrations.

Let $T_\ast(\pi)$ be the category based left $\pi$-spaces and $\pi$-equivariant maps; this category has a zero object, i.e., the space $\ast$
consisting of a single point. Declare a map to be
a weak equivalence (fibration) if and only if it is a weak equivalence (resp.~fibration)  upon application 
of the forgetful functor $T_\ast(\pi) \to T$. A map is a cofibration if it possesses the left lifiting property with respect to the acyclic fibrations.
Then with respect to these choices, $T_\ast(\pi)$ is a model category by application of \cite[thm.~11.6.1]{Hirschhorn}.
Note that the cofibrant objects of $T_\ast(\pi)$ are retracts of objects built up from a point by attaching free cells, where a free cell
of dimension $j$ is $D^j \times \pi$. In particular, a free $\pi$-CW complex (in the base sense) is cofibrant. Note that
every object is fibrant.

For cofibrant objects $X,Y\in T_\ast(\pi)$, let 
\[
[X,Y]_{\pi} = \hom_{\text{ho}T_\ast(\pi)}(X,Y)\, .
\]
Note that $Q_{\Bbb Z_2}(Y)$ has the structure of an object of $T_\ast(\pi\times \Bbb Z_2)$.
Suppose in addition $X$ has the structure of a based $(\pi\times \Bbb Z_2)$-CW complex which is $\pi$-free after forgetting the $\Bbb Z_2$-action.
In this case, we set
\[
\{X,Y\}_{\pi \times \underline{\Bbb Z_2}} := \pi_0(\map_{T_\ast(\pi)}(X,Q_{\Bbb Z_2}(Y))^{\Bbb Z_2})\, ,
\]
where $\map_{T_\ast(\pi)}(A,B)$ denotes the mapping space of based equivariant maps $A\to B$.

\begin{proof}[Proof of Addendum \ref{bigadd:Hopf-properties}]
With respect to the above conventions, the equivariant stable Hopf invariant is defined in the same way
as when $\pi$ is trivial: For a cofibrant $\pi$-space $Y$, the tom Dieck fiber sequence
\[
\Sigma^\infty Y_{h\Bbb Z_2} \to (\Sigma^\infty_{\Bbb Z_2} Y)^{\Bbb Z_2} \to \Sigma^\infty Y^{\Bbb Z_2} 
\]
is a fiber sequence of naive spectra with $\pi$-action which splits equivariantly making use of the inclusion $Y^{\Bbb Z_2} \to Y$.
If we  set $Y = B \smsh B$ where $B\in T_{\ast}(\pi)$ is a cofibrant object, we see
then the construction of \S\ref{sec:stable-Hopf} yields 
\[
H^\pi \:\{A,B\}_{\pi} \to \{A,D_2(B)\}_{\pi}\, .
\]
The  proof that properties (i)-(iv) hold is exactly the same as the proof of Theorem \ref{bigthm:Hopf-properties} above,
where the maps and homotopy classes are to be interpreted equivariantly.

It remains to show that in the metastable range, a homotopy class $f\in \{A,B\}_\pi$ destabilizes if and only if $H^\pi(f) =0$.
We will provide a crude sketch of the argument and leave the details to the reader.
The Snaith splitting is natural,\footnote{More precisely,  the
proofs that appear in \cite{Cohen_snaith} and \cite{Goodwillie_calc3} are natural.}
so the weak homotopy equivalence
\[
\Sigma^\infty Q(B) @> \simeq >>\textstyle\bigvee\limits_{n\ge 1}\Sigma^\infty D_n(B) 
\]
is also  $\pi$-equivariant.
It follows that a $\pi$-equivariant version $h^\pi$ of the second 
Segal-Snaith Hopf invariant $h$ is defined. Moreover, the sequence
\[
B \to Q(B) @> h^\pi >> Q(D_2(B))
\]
is an equivariant fiber sequence in the metastable range by the Blakers-Massey excision theorem, where
the map $h^\pi$ is induced by projection onto the second summand of the Snaith splitting.
Then the equivariant stable Hopf invariant $h^\pi$ is the the effect of applying $[A,{-}]_\pi$ to the map
of the same name.
We infer
that $h^\pi$ yields the total obstruction to destabilization in the metastable range. Lastly,
the proof we gave that $H = h$ is valid in the equivariant setting, and we conclude that $H^\pi = h^\pi$.
\end{proof} 

\section{Proof of Theorem \ref{bigthm:unique}} \label{sec:unique}
Consider a natural transformation
\[
\lambda\: \{A,B\} \to \{A,D_2(B)\}
\]
which satisfies properties (i)-(iii). 

By naturality and the Snaith splitting, $\lambda$ is induced by a natural transformation of spectrum-valued functors
\[
\textstyle\bigvee\limits_{n\ge 1}\lambda_n\: \textstyle\bigvee\limits_{n\ge 1} \Sigma^\infty D_n(B)  \to \Sigma^\infty D_2(B)\, .
\]
Since the functor $B\mapsto \Sigma^\infty D_n(B)$ is homogeneous of degree $n$, it follows from standard Goodwillie calculus arguments
\cite{Goodwillie_calc3}, that the natural transformation  $\lambda_n$ is homotopically trivial for $n > 2$. Moreover,
by property (a), $\lambda_1$ is trivial. Hence, it suffices to determine $\lambda_2$. Note that for the Segal-Snaith Hopf invariant,  $h_2$ is the identity
\cite[app.~B]{Kuhn_diag}.

For homotopy functors $F,G$ from based spaces to spectra, let $\nat(F,G)$ be the abelian group of homotopy classes of natural transformations
from $F$ to $G$; when $F = G$, this is a ring with respect to composition. Let  $\{S^0,S^0\}_{\Bbb Z_2}$ 
be the ring of homotopy classes of equivariant stable self maps of $S^0 \smsh {E\Bbb Z_2}_+$.  
 Again by Goodwillie calculus, the ring homomorphism  
 \[
 \{S^0,S^0\}_{\Bbb Z_2} \to \nat(D_2({-}),D_2({-}))
 \] 
induced by $\alpha\mapsto \alpha \smsh_{\Bbb Z_2} 1$
 is an isomorphism. In particular $\hat A(\Bbb Z_2) = \{S^0,S^0\}_{\Bbb Z_2}$ acts on $\nat(D_2({-}),D_2({-}))$.
 
 Let $\theta\in  \{S^0,S^0\}_{\Bbb Z_2}$ correspond to $\lambda_2$; then $\lambda_2 = \theta h_2$.
 I claim that $\theta$ is a unit. To see this, represent $\theta$ by an equivariant self map $S^0 \smsh {E\Bbb Z_2}_+ \to S^0 \smsh {E\Bbb Z_2}_+$
 and let $W_\theta$ be its homotopy fiber. Then the diagram
 \[
 \xymatrix{
 B \ar[r] \ar@{=}[d] & Q(B) \ar[r]^(.4){h} \ar@{=}[d] &  Q(D_2(B)) \ar[d]^\theta\\
 B \ar[r] & Q(B) \ar[r]_(.4){\lambda}  & Q(D_2(B)
 }
 \]
 homotopy commutes where the horizontal arrows form the EHP sequences for $h$ and $\lambda$.
 The homotopy fiber of the right vertical map is given by the infinite loop space corresponding to the spectrum
 $W_{\theta} \smsh_{h\Bbb Z_2} (B\smsh B)$.  By property (c), the latter is $(3r)$-connected whenever $B$ is $r$-connected.
 If we let $B = S^{r+1}$ then 
 \[
 W_{\theta} \smsh_{h\Bbb Z_2} (S^{r+1}\smsh S^{r+1}) \simeq \Sigma^{r+1} W_{\theta} \smsh_{h\Bbb Z_2} S^{(r+1)\alpha}
\]
Suppose that $W_\theta$ is $s$-connected. If $r$ is odd, then $\Bbb Z_2$ acts trivially on the reduced homology of 
$S^{(r+1)\alpha}$. It follows that $\Sigma^{r+1}W_{\theta} \smsh_{h\Bbb Z_2} S^{(r+1)\alpha}$ is $(s+2r+2)$-connected.
Hence $s+2r+2 \ge 3r$ for $r$-odd. Since $r$ is arbitrary, it follows that $W_{\theta}$ is weakly contractible. 
We infer that $\theta$ is a unit, proving the claim.

By the Segal conjecture \cite{Carlsson_Segal}, the ring $\{S^0,S^0\}_{\Bbb Z_2}$ is canonically isomorphic to the completed Burnside ring
$\hat A(\Bbb Z_2)$.  
Consider the commutative diagram
\[
\xymatrix{
\hat A(\Bbb Z_2) \ar[r]^(.4){\cong} \ar[d] &  \{S^0,S^0\}_{\Bbb Z_2} \ar[r]^(.4){\cong}  \ar[d] & \nat(D_2(B),D_2(B)) \ar[d]^{\pi^\ast_B}  \\
\hat A(e) \ar[r]_(.4){\cong} & \{S^0,S^0\} \ar[r]_(.4){\cong} & \nat(B\smsh B ,D_2(B))
}
\]
where the right vertical arrow is induced by the evident map $\pi_B\: \: B\smsh B\to D_2(B)$.  The lower right arrow 
is an isomorphism by Goodwillie calculus since a natural transformation $B\smsh B \to D_2(B)$ is determined by
an equivariant map $\Sigma^\infty {\Bbb Z_2}_+ \to S^0$ and the latter is the same as specifying a non-equivariant map
$S^0 \to S^0$. We observe  that $B\mapsto \pi_B$ corresponds
to the identity element of $\{S^0,S^0\}  = \Bbb Z$.

 Let $p_i\: B_+ \smsh B_+ = (B\times B)_+ \to B$ denote the projections $i=1,2$. Let $q_{B_+}\: B_+ \smsh B_+ \to D_2(B)$ be the composition
\[
B_+ \smsh B_+ @>\pi_{B_+} >> D_2(B_+) \to D_2(B)\, .
\]
By the Cartan formula for $\lambda$ and $H$, we have 
\[
q_{B_+} = \lambda(p_1 + p_2) = \theta h(p_1 + p_2) = \theta q_{B_+}\, .
\]
Since $\{B\smsh B,D_2(B)\} \to \{B_+\smsh B_+,D_2(B)\}$ is injective, and the image of $q_{B_+}$ is $\pi_B$, we infer that
$\pi_B = \theta \pi_B$. Consequently, $\theta \in \hat A(\Bbb Z_2)^{\times}$ augments to the identity element.
We conclude that $\theta$ lies in the kernel $K$ of the homomorphism 
$\hat A(\Bbb Z_2)^{\times} \to \hat A(e)^{\times}$. 

Conversely, it is clear that any element $\theta\in K$ determines a natural transformation $\lambda$ saitsfying the properties (a)-(c).
 \qed

\subsection{Note added in proof} The EHP property (c) is a consequence of properties (a) and (b). This follows from the observation
that the elements of $ \{S^0,S^0\}_{\Bbb Z_2}$ which augment to units of $ \{S^0,S^0\} $ are themselves units.



\begin{thebibliography}{10}
\bibitem{Adams_infinite}
John~Frank Adams, \emph{Infinite loop spaces}, Annals of Mathematics Studies,
  vol. No. 90, Princeton University Press, Princeton, NJ; University of Tokyo
  Press, Tokyo, 1978.

\bibitem{Boardman-Steer}
J.~M. Boardman and B.~Steer, \emph{On {H}opf invariants}, Comment. Math. Helv.
  \textbf{42} (1967), 180--221.

\bibitem{Bouc_Burnside-survey}
Serge Bouc, \emph{Burnside rings}, Handbook of algebra, {V}ol. 2, Handb.
  Algebr., vol.~2, Elsevier/North-Holland, Amsterdam, 2000, pp.~739--804.

\bibitem{Carlsson_Segal}
Gunnar Carlsson, \emph{Equivariant stable homotopy and {S}egal's {B}urnside
  ring conjecture}, Ann. of Math. (2) \textbf{120} (1984), no.~2, 189--224.

\bibitem{Caruso-Cohen-May-Taylor}
J.~Caruso, F.~R. Cohen, J.~P. May, and L.~R. Taylor, \emph{James maps, {S}egal
  maps, and the {K}ahn-{P}riddy theorem}, Trans. Amer. Math. Soc. \textbf{281}
  (1984), no.~1, 243--283.

\bibitem{Cohen-May-Taylor}
F.~R. Cohen, J.~P. May, and L.~R. Taylor, \emph{Splitting of certain spaces
  {$CX$}}, Math. Proc. Cambridge Philos. Soc. \textbf{84} (1978), no.~3,
  465--496.

\bibitem{Cohen_snaith}
Ralph~L. Cohen, \emph{Stable proofs of stable splittings}, Math. Proc.
  Cambridge Philos. Soc. \textbf{88} (1980), no.~1, 149--151.

\bibitem{Crabb-Ranicki}
Michael Crabb and Andrew Ranicki, \emph{The geometric {H}opf invariant and
  surgery theory}, Springer Monographs in Mathematics, Springer, 2017.

\bibitem{Goodwillie_calc3}
Thomas~G. Goodwillie, \emph{Calculus. {III}. {T}aylor series}, Geom. Topol.
  \textbf{7} (2003), 645--711.

\bibitem{Hirschhorn}
Philip~S. Hirschhorn, \emph{Model categories and their localizations},
  Mathematical Surveys and Monographs, vol.~99, American Mathematical Society,
  Providence, RI, 2003.

\bibitem{Hopf}
Heinz Hopf, \emph{\"uber die {A}bbildungen der dreidimensionalen {S}ph\"are auf
  die {K}ugelfl\"ache}, Math. Ann. \textbf{104} (1931), no.~1, 637--665.

\bibitem{James-suspension-triad}
I.~M. James, \emph{On the suspension triad}, Ann. of Math. (2) \textbf{63}
  (1956), 191--247.

\bibitem{Koschorke-Sanderson}
Ulrich Koschorke and Brian Sanderson, \emph{Self-intersections and higher
  {H}opf invariants}, Topology \textbf{17} (1978), no.~3, 283--290.

\bibitem{Kuhn_James-Hopf}
Nicholas~J. Kuhn, \emph{The geometry of the {J}ames-{H}opf maps}, Pacific J.
  Math. \textbf{102} (1982), no.~2, 397--412.

\bibitem{Kuhn_diag}
\bysame, \emph{Stable splittings and the diagonal}, Homotopy methods in
  algebraic topology ({B}oulder, {CO}, 1999), Contemp. Math., vol. 271, Amer.
  Math. Soc., Providence, RI, 2001, pp.~169--181.

\bibitem{May_equivariant}
L.~G. Lewis, Jr., J.~P. May, M.~Steinberger, and J.~E. McClure,
  \emph{Equivariant stable homotopy theory}, Lecture Notes in Mathematics, vol.
  1213, Springer-Verlag, Berlin, 1986, With contributions by J. E. McClure.

\bibitem{May_approximation}
J.~P. May, \emph{Applications and generalizations of the approximation
  theorem}, Algebraic topology, {A}arhus 1978 ({P}roc. {S}ympos., {U}niv.
  {A}arhus, {A}arhus, 1978), Lecture Notes in Math., vol. 763, Springer,
  Berlin, 1979, pp.~38--69.

\bibitem{Milgram}
R.~James Milgram, \emph{Unstable homotopy from the stable point of view},
  Lecture Notes in Mathematics, vol. Vol. 368, Springer-Verlag, Berlin-New
  York, 1974.

\bibitem{Quillen-homotopical}
Daniel~G. Quillen, \emph{Homotopical algebra}, Lecture Notes in Mathematics,
  vol. No. 43, Springer-Verlag, Berlin-New York, 1967.

\bibitem{Ranicki}
Andrew Ranicki, \emph{The algebraic theory of surgery. {II}. {A}pplications to
  topology}, Proc. London Math. Soc. (3) \textbf{40} (1980), no.~2, 193--283.

\bibitem{Segal1973}
Graeme Segal, \emph{Configuration-spaces and iterated loop-spaces}, Inventiones
  Mathematicae \textbf{21} (1973), 213--221.

\bibitem{Segal_hopf}
\bysame, \emph{Operations in stable homotopy theory}, New developments in
  topology ({P}roc. {S}ympos. {A}lgebraic {T}opology, {O}xford, 1972), London
  Math. Soc. Lecture Note Ser., vol. No. 11, Cambridge Univ. Press, London-New
  York, 1974, pp.~105--110.

\bibitem{Snaith-stable}
V.~P. Snaith, \emph{A stable decomposition of {$\Omega \sp{n}S\sp{n}X$}}, J.
  London Math. Soc. (2) \textbf{7} (1974), 577--583.

\end{thebibliography}

\end{document}